\documentclass[12pt]{article}
\usepackage{amsmath,amssymb,amsbsy,amsfonts,amsthm,latexsym,
         amsopn,amstext,amsxtra,euscript,amscd,amsthm}

\newtheorem{lem}{Lemma}

\newtheorem{thm}{Theorem}

\def\\{\cr}
\def\({\left(}
\def\){\right)}
\def\[{\left[}
\def\]{\right]}
\def\<{\langle}
\def\>{\rangle}

\def\Z{{\mathbb Z}}

\def\Q{{\mathbb Q}}
\def\K{{\mathbb K}}
\def\L{{\mathbb L}}
\def\M{{\mathbb M}}

\begin{document}

\title{Residue Classes Having Tardy Totients}

\author{{\sc John B.~Friedlander} \\
{Department of Mathematics, University of Toronto} \\
{Toronto, Ontario M5S 3G3, Canada} \\
{\tt frdlndr@math.toronto.edu} \and
{\sc Florian~Luca} \\
{Instituto de Matem{\'a}ticas}\\
{Universidad Nacional Aut\'onoma de M{\'e}xico} \\
{C.P. 58089, Morelia, Michoac{\'a}n, M{\'e}xico} \\
{\tt fluca@matmor.unam.mx}}

\pagenumbering{arabic}

\date{\today}

\maketitle

\begin{abstract}
We show, in an effective way, that there exists a sequence of
congruence classes $a_k\pmod {m_k}$ such that the minimal solution
$n=n_k$ of the congruence $\phi(n)\equiv a_k\pmod {m_k}$ exists and
satisfies $\log n_k/\log m_k\to\infty $ as $k\to\infty$. Here,
$\phi(n)$ is the Euler function. This answers a question raised in
\cite{FS}. We also show that every congruence class containing an
even integer contains infinitely many values of the Carmichael
function $\lambda(n)$ and the least such $n$ satisfies $n\ll
m^{13}$.
\end{abstract}

\section{Introduction}

Let $\phi(n)$ be the Euler function of $n$. The number $\phi(n)$ is
also referred to as the {\it totient} of $n$, and so the values of
the Euler function are referred to as {\it totients}. There are a
number of papers in the literature dealing with the question of
which arithmetical progressions contain infinitely many totients.
Since all totients $>1$ are even, it follows that the only
arithmetical progressions that can contain infinitely many totients
are the ones which contain even integers. Dence and Pomerance
\cite{DP} established that if a congruence class $a\pmod m$ contains
at least one multiple of $4$, then it contains infinitely many
totients. Later, Ford, Konyagin and Pomerance \cite{FKP} gave a
characterization of which arithmetical progressions consisting
entirely of numbers congruent to $2\pmod 4$ contain infinitely many
values of the Euler function. They also showed that the union of all
such residue classes which are totient-free (i.e., do not contain
totients) has asymptotic density $1/4$ (as does the entire
progression $2\pmod 4$), establishing in this way that almost all
integers which are $\equiv 2\pmod 4$ are in a residue class that is
totient-free. See also \cite{N} for related results.

\medskip

Recently, the first--named author and Shparlinski \cite{FS} looked
at progressions $a\pmod m$ containing totients and asked about the
size of the smallest totient, denoted by $N(a,m)$, in such a
progression. Under the restriction $\gcd(2a,m)=1$ (note that in this
case $m$ is necessarily odd), they established that $N(a,m)\le
m^{3+o(1)}$ holds uniformly in $a$ and $m$ as $m\to\infty$. When $m$
is prime, the exponent $3$ can be improved to 2.5. They asked
whether a similar result holds, perhaps with an exponent larger than
$3$, provided that one eliminates the assumption that $2a$ and $m$
are coprime. Specifically, Open Question 7 in \cite{FS} is the
following.

\medskip

\noindent {\bf Open Question.} {\it Does there exist a positive
constant $A$ such that, for every pair of integers $a$ and $m$, if
there exists an integer $n$ with $\phi(n)\equiv a\pmod m$, then
there exists such an integer with $n<m^A$?}

\medskip

Here, we give a negative answer to the above question. In what
follows, we write $\log x$ for the natural logarithm of $x$.  Our
result is the following.

\begin{thm}
\label{thm:1} There exists a sequence $a_k\pmod {m_k}$ of
arithmetical progressions with $m_k\to\infty$ such that $N(a_k,m_k)$
exists and satisfies
\begin{equation}
\label{eq:maineq1} N(a_k,m_k)\ge m_k^{(c_0+o(1))(\log\log
m_k/\log\log\log m_k)^{1/2}}
\end{equation}
as $k\to\infty$, where $c_0>0$ is an absolute constant.
\end{thm}

The value $c_0=2/5$ is acceptable in Theorem \ref{thm:1}.  Under the
hypothesis of the Generalized Riemann Hypothesis for certain
algebraic number fields our argument yields the stronger lower bound
\begin{equation}
\label{eq:maineq2} N(a_k,m_k)\ge m_k^{(c_1+o(1))(\log m_k/\log\log
m_k)^{1/2}}\ ,
\end{equation}
as $k\to\infty$, where $c_1=2^{-1/2}$.

\medskip

While the Euler function $\phi(n)$ measures the size of the
multiplicative group $\left(\Z/n\Z\right)^\times$, the exponent of
this group (i.e., largest order of elements) is called the {\it
Carmichael function\/} of $n$ and is denoted by $\lambda(n)$. Since
$(-1)^2\equiv 1\pmod n$, it follows that $\lambda(n)$ is even for
all $n\ge 3$. While the existence of totients in arithmetical
progressions has received interest, the analogous problem concerning
the presence of values of the Carmichael function in residue classes
seems not to have been investigated. Although it is often the case
that problems for the latter function are the more difficult of the
two, in case of the question of which progressions occur, the
Carmichael function is the easier. As for the question answered for
$\phi$ by Theorem~\ref{thm:1}, the answer for $\lambda$ is the
opposite.

\begin{thm}
\label{thm:2} If the arithmetical progression $a\pmod m$ contains an
even number, then it contains infinitely many values of the
Carmichael function. Furthermore, writing $L(a,m)$ for the least
integer $n$ such that $\lambda(n)\equiv a\pmod m$, we have
$L(a,m)\ll m^{13}$.
\end{thm}

\section{Proof of Theorem \ref{thm:1}}

We start with the following lemma.

\begin{lem}
\label{lem:1} Let $k\ge 2$ be an integer and $a=2q$, where $q>3$ is
prime. Then $f_k(X)=X^k-X^{k-1}-a\in \Q[X]$ is irreducible.
\end{lem}

\begin{proof}
Assume for a contradiction that $f_k(X)=g(X)h(X)$, where $g(X)$ and
$h(X)$ are monic, with integer coefficients and positive degrees.
Reducing modulo $q$, we get $g(X)h(X)\equiv X^{k-1}(X-1)\pmod q$,
therefore, up to relabeling the polynomials $g(X)$ and $h(X)$, we
may assume that $g(X)=X^{\alpha}(X-1)\pmod q$ and
$h(X)=X^{\beta}\pmod q$, where $\alpha+\beta=k$. If $\alpha\beta>0$,
then $g(0)\equiv h(0)\equiv 0\pmod q$. Thus, the last coefficient of
each of $g(X)$ and $h(X)$ is a multiple of $q$. This implies that
the last coefficient of $f_k(X)$, which is $f_k(0)=-a=-2q=g(0)h(0)$
is a multiple of $q^2$, which is a contradiction. This shows that
$\alpha\beta=0$ and since both $g(X)$ and $h(X)$ have positive
degrees, it follows that $\alpha=0$. Hence, $g(X)\equiv X-1\pmod q$,
therefore $g(X)$ is linear. Write $g(X)=X-x_0$ for some $x_0\in\Z$
with $x_0\equiv 1\pmod q$. Then $x_0$ is a root of $f_k(X)$,
therefore $2q=a=x_0^{k-1}(x_0-1)$. Since $q\mid x_0-1$, it follows
that $x_0^{k-1}\mid 2$. If $k\ge 3$, we get $x_0=\pm 1$, therefore
$2q=x_0^{k-1}(x_0-1)\in \{0,\pm 2\}$, which is impossible. If $k=2$,
we then get that $x_0\in \{\pm 1,\pm 2\}$, therefore
$2q=a=x_0(x_0-1)\in \{0,2,6\}$, which is again impossible. Hence,
$f_k(X)\in \Q[X]$ is irreducible.
\end{proof}

We are now ready to prove Theorem \ref{thm:1}. Let $L$ be any large
even integer. We let $q>L$ be a prime congruent to $1$ modulo $3$,
put $a=2q$ and let again $f_k(X)=X^k-X^{k-1}-a\in \Q[X]$ be the
polynomials that appear in Lemma \ref{lem:1} for $k=2,3,\ldots,L$.
They are all irreducible by Lemma \ref{lem:1}. For each
$k=2,\ldots,L$, we let $\theta_k$ be some root of $f_k(X)$. We write
$\theta_k^{(1)},\ldots,\theta_k^{(k)}$ for all the conjugates of
$\theta_k$ with the convention that $\theta_k^{(1)}=\theta_k$. We
put $\K_{k}^{(j)}=\Q[\theta_k^{(j)}]$ for $j=1,\ldots,k$. We also
put ${\overline{\K_k}}=\Q[\theta_k^{(1)},\ldots,\theta_k^{(k)}]$ for
the splitting field of $f_k(X)$ and $\M=\Q[\theta_k^{(j)}:2\le k\le
L,1\le j\le k]$ for the splitting field of $\prod_{k=2}^L f_k(X)$.
Our first objective is to insure that we can choose an appropriate
$a$ of the desired form which is not too large such that
$G={\text{\rm Gal}}(\M/\Q)=\prod_{k=2}^L S_k$, where we write $S_m$
for the symmetric group on $m$ letters.

\medskip

We start by computing the discriminant of $\K_k=\K_k^{(1)}$. It is
well-known (see \cite{Sw}) that if $f(X)=X^n+AX^s+B\in \Q[X]$, where
$1\le s\le n-1$, then the discriminant of the polynomial $f(X)$ is
$$
\Delta(f)=(-1)^{n (n-1)/2}
B^{s-1}\left(n^nB^{n-s}+(-1)^{n-1}s^s(n-s)^{n-s}A^n\right).
$$
Thus, if we put $\Delta_k=\Delta(f_k)$, we then have that
$n=k,~s=k-1,~A=-1$ and $B=-a$, therefore
$$
\Delta_k=(-1)^{(k-1)(k+2)/2}a^{k-2}(ak^k+(k-1)^{k-1}).
$$
Let ${\cal M}$ be the set of all prime numbers $p\ne 2,~q$ that can
appear as divisors of $\gcd(\Delta_j,\Delta_k)$ for some $2\le
j<k\le L$. If $p\ne 2,q$ is a prime factor of $\Delta_j$, then
$p\mid aj^j+(j-1)^{j-1}$. Since $j$ and $j-1$ are coprime, we see
that $p\nmid j$, therefore $j$ is invertible modulo $p$. Thus,
$a\equiv -(j-1)^{j-1}j^{-j}\pmod p$. If additionally $p\mid
\Delta_k$, then also $a\equiv -(k-1)^{k-1}k^{-k}\pmod p$. Thus,
$$
-(j-1)^{j-1}j^{-j}\equiv -(k-1)^{k-1}k^{-k}\pmod p
$$
leading to $p\mid (j-1)^{j-1}k^k-j^j(k-1)^{k-1}$. Thus, writing
$$
M=\prod_{2\le j<k\le L} \left((j-1)^{j-1}k^k-j^j(k-1)^{k-1}\right),
$$
we conclude that
$$
\prod_{p\in {\cal M}}p\mid M.
$$
Note that $M>0$ since the function $x\mapsto (x-1)^{x-1} x^{-x}$ is
decreasing for $x\ge 2$. Note further that $M$ consists of a product
of $\binom{L-1}{2}$ factors none of which exceeds
$L^L(L-1)^{L-1}<L^{2L}$. Thus,
$$
M\le \left(L^{2L}\right)^{\binom{L-1}{2}}<L^{L^3}.
$$
Thus, writing $\omega(m)$ for the number of distinct prime factors
of $m$ and using the known fact that $\omega(m)\le (1+o(1))\log
m/\log\log m$ as $m\to\infty$, we get that
\begin{eqnarray}
\label{eq:P} \#{\cal M} & \le & \omega(M)\le (1+o(1))\frac{\log
M}{\log\log M}  \le (1+o(1))\frac{L^3\log L}{\log(L^3\log
L)}\nonumber\\
& \le & (1/3+o(1))L^3
\end{eqnarray}
as $L\to\infty$.

\medskip

Let $T>L$ be some large number to be determined later. We search for
a value of $a=2q$, where $L<q\le T$ is prime congruent to $1$ modulo
$3$, such that for each $k=2,\ldots,L$ there exists a prime
$p_k\|\Delta_k$ such that $p_k$ does not divide any of the
$\Delta_j$ for $j\ne k$ in $\{2,\ldots,L\}$. To conclude that such a
$q$ exists, assume the contrary. Then for each prime $q\in (L,T]$
congruent to $1$ modulo $3$, there exists $k\in \{2,\ldots,L\}$ such
that $ak^k+(k-1)^{k-1}=sy$ holds with some square-free number $s$
dividing $2M$ and some positive integer $y$ which is square-full.
Recall that a positive integer $m$ is square-full if $p^2\mid m$
whenever $p$ is a prime factor of $m$. The number of choices for $s$
is at most $2^{\omega(M)+1}$. For a large positive real number $x$,
the number of square-full numbers $y\le x$ does not exceed
$3x^{1/2}$ (see, for example, Th\'eor\`eme 14.4 in \cite{Iv}). Since
$ak^k+(k-1)^{k-1}\le 2aL^L\le 4TL^L$, it follows that $y$ can be
chosen in at most $3(4TL^L)^{1/2}$ ways. Hence, the number of
possibilities for $a\in (L,T]$ such that there exists $k\in
\{2,\ldots,L\}$ with $ak^k+(k-1)^{k-1}=sy$, where $s\mid 2M$ and $y$
is square-full is
$$
\le 3(L-1)2^{\omega(M)+1}(4TL^L)^{1/2}.
$$
Since there are $\pi(T;3,1)-\pi(L;3,1)$ primes $q\in (L,T]$, we get
that a desired choice for $a$ is possible once
\begin{equation}
\label{eq:piTL} \pi(T;3,1)-\pi(L;3,1)>
3(L-1)2^{\omega(M)+1}(4TL^L)^{1/2}.
\end{equation}
Here, we used, as usual, for coprime integers $a$ and $b$ the
notation $\pi(x;b,a)$ for the number of primes $p\le x$ congruent to
$a$ modulo $b$. Let $T_L$ be the minimal positive integer satisfying
inequality \eqref{eq:piTL}. Estimate \eqref{eq:P} together with
standard estimates for primes in arithmetical progressions shows
that
$$
T_L\le \exp((2(\log 2)/3+o(1))L^3)\qquad {\text{\rm as}}~L\to\infty.
$$
Since $p_k\ne 2,q$ exactly divides $\Delta_k$, it follows that
$p_k\mid \Delta_{\K_k}$, where for a field $\L$ we put $\Delta_\L$
for its discriminant. Furthermore, Theorem 1.1 in \cite{CMS} and the
remarks following it show that if $f(X)=X^n+AX^s+B$ is irreducible,
with integer coefficients and satisfies the conditions
$\gcd(n,As)=\gcd(A(n-s),B)=1$ and there is a prime divisor $q$ of
$B$ such that $q\|b$, then the Galois group $G(f)$ of $f(X)$ over
$\Q$ is doubly transitive. When $f(X)=f_k(X)$, we have that
$n=k,~s=k-1,A=-1$ and $B=-a=-2q$, therefore all three conditions
$\gcd(n,As)=\gcd(A(n-s),B)=1$ and $q\|B$ are satisfied. Thus, the
Galois group of $f_k(X)$ over $\Q$ is doubly transitive. The remarks
following Theorem 1.1 in \cite{CMS} show that if furthermore there
exists a prime $p_k$ not dividing $\gcd(k-1,2q)$ (which for us
equals $1$ or $2$ since $q>L$) such that $p_k$ exactly divides
$\Delta_k$, then the Galois group of $f_k(X)$ over $\Q$ contains a
transposition and is, in particular, the full symmetric group $S_k$.
Thus, we have showed that ${\text{\rm Gal}}(\K_k/\Q)=S_k$.

\medskip

We are now ready to show that $G=\prod_{k=2}^L S_k$. Since
${\text{\rm Gal}}(\K_k/\Q)=S_k$ for each $k\in \{2,\ldots,L\}$, it
suffices to show the family of fields
$\{{\overline{\K_k}}:k=2,\ldots.L\}$ consists of {\it linearly
disjoint} fields, namely that if $\{k_1,\ldots,k_{s+1}\}$ is any
subset of $\{2,\ldots,L\}$ with $k_{s+1}\not\in \{k_1,\ldots,k_s\}$,
then
\begin{equation}
\label{eq:compo} {\overline{\K_{k_1}}} {\overline{\K_{k_2}}}\cdots
{\overline{\K_{k_s}}}\cap {\overline{\K_{k_{s+1}}}}=\Q.
\end{equation}
Well, let us denote by $\L$ the field appearing in the left hand
side of the above equality and assume that it is not $\Q$. Since
$\L$ is an intersection of normal extensions of $\Q$, it follows
that it is itself a normal extension of $\Q$. Furthermore, if $p$ is
a prime dividing $\Delta_{\L}$ then $p$ divides both
$\Delta_{k_{s+1}}$ and $\prod_{i=1}^s \Delta_{k_i}$. In particular,
$p\ne p_{k_{s+1}}$, which shows that $\L\not\subseteq
{\overline{\K_{k_{s+1}}}}$. Since $\L$ is a normal extension of $\Q$
properly contained in ${\overline{\K_{k_{s+1}}}}$, it follows that
${\text{\rm Gal}}(\K_{k_{s+1}}/\L)$ is a normal proper subgroup of
${\text{\rm Gal}}(\K_{k_{s+1}}/\Q)=S_k$. The only such subgroup is
$A_k$ and, by Galois correspondence, we get that
$\L=\Q({\sqrt{\Delta_{\K_{k_{s+1}}}}})$. But this last quadratic
field has the property that $p_{k_{s+1}}$ divides its discriminant,
whereas we have established that the prime $p_{k_{s+1}}$ cannot
divide the discriminant of the field $\L$. This contradiction shows
that indeed equality \eqref{eq:compo} holds, which establishes our
claim on the structure of the Galois group of $\M$.

\medskip

Let ${\cal C}$ be a conjugacy class of $G$ containing an element
$\sigma=(\sigma_2,\ldots,\sigma_L)$ where $\sigma_k\in S_k$ has no
fixed points for all $k=2,\ldots,L-1$, but $\sigma_L\in S_L$ is the
identical permutation. By Chebotarev's Density Theorem, a positive
proportion of all the primes $p$ not dividing the discriminant of
$\M$ have the property that their Frobenius ${\text{\rm Frob}}_p$ is
in the conjugacy class ${\cal C}$. If $p$ is such a prime, then
$f_k(X)\pmod p$ has no root modulo $p$ for any $k=2,\ldots,L-1$,
while $f_L(X)\pmod p$ splits in linear factors modulo $p$.

\medskip

We now let $x$ be a large positive real number. We need a lower
bound for the number of primes $p\le x$ in the conjugacy class
${\cal C}$ of $G$. To this end, we use the following result which is
implicit in the work of Lagarias, Montgomery and Odlyzko on the
least prime ideal in the Chebotarev Density Theorem \cite{LMO}.

\begin{lem}
\label{lem:2} Let $\M$ be a Galois extension of $\Q$ of discriminant
$\Delta_{\M}$ having Galois group $G$. Let ${\cal C}$ be a conjugacy
class in $G$ and define
$$
\pi_{\cal C}(x,\M/\Q)=\#\{p\le x:p\nmid \Delta_{\M},~{\text{\rm
Frob}}_p\in {\cal C}\}.
$$
There exist absolute constants $A_1$ and $A_2$ such that if
$x>|\Delta_{\M}|^{A_1}$, then
\begin{equation}
\label{eq:initial} \pi_{\cal C}(x,\M/\Q)\gg \frac{\#{\cal
C}}{\#G}\frac{x^{1/5}}{|\Delta_{\M}|^{A_2}}.
\end{equation}
\end{lem}

\begin{proof}
Let
$$
{\cal P}({\cal C})=\{p~:~p\nmid \Delta_{\M},~{\text{\rm Frob}}_p\in
{\cal C}\}.
$$
Inequality (6.9) in \cite{LMO} shows that there exist positive
absolute constants $B_1,B_2,B_3$ and $B_4$ such that
\begin{eqnarray}
\label{eq:someineq} \sum_{\substack{p\in {\cal P}({\cal C})\\ p\le
x^{10}}} (\log p){\widehat{k_2}}(p) & \ge &
\frac{x^2}{10}\frac{\#{\cal C}}{\#G} \min\{1,(1-\beta_0)\log
x\}-B_1x^{7/4}\log |\Delta_{\M}|\nonumber\\
& - & B_2x^2(1-\beta_0)^{B_3\log
x/\log|\Delta_{\M}|}\log|\Delta_{\M}|,
\end{eqnarray}
where $\beta_0\in (0,1)$ satisfies $1-\beta_0>|\Delta_{\M}|^{-B_4}$
(see Corollary 5.2 on page 290 in \cite{LMO}),  and
${\widehat{k_2}}(u)$ is a function whose range is in the interval
$(0,1)$ (see formula (3.7) on page 284). The argument on the top of
page 294 in \cite{LMO}, shows that if we choose
$x>|\Delta_{\M}|^{B_5}$ for a suitable absolute constant $B_5>0$,
then the first term in the right hand side of inequality
\eqref{eq:someineq} dominates. Thus, inequality \eqref{eq:someineq}
implies that
$$
\pi_{\cal C}(x^{10},\M/\Q)\log x\gg \frac{\#{\cal C}}{\#G}
x^2\min\{1,(1-\beta_0)\log x\}\gg \frac{\#{\cal
C}}{\#G}\frac{x^2\log x}{|\Delta_{\M}|^{B_4}},
$$
which in turn implies that
$$
\pi_{\cal C}(x^{10},\M/\Q)\gg \frac{\#{\cal
C}}{\#G}\frac{x^2}{|\Delta_{\M}|^{B_4}},
$$
which is what we wanted to prove with $A_1=10B_5$ and $A_2=B_4$.
\end{proof}

For an algebraic number field $\L$, we write $d_{\L}$ for its
degree. In order to apply the above Lemma \ref{lem:2}, we need upper
bounds on $\Delta_{\M}$. Clearly,
\begin{eqnarray}
\label{eq:degree} d_{\M} & = & \#G=\prod_{k=2}^L k!\le \prod_{k=2}^L
k^k=\exp\left(\sum_{k=2}^L k\log
k\right)\nonumber\\
& = & \exp\left((1/2+o(1))L^2\log L\right),\qquad {\text{\rm
as}}~L\to\infty.
\end{eqnarray}
As for $|\Delta_{\M}|$, we use recursively the fact that if $\K\cap
\L=\Q$, then
\begin{equation}
\label{eq:delta} |\Delta_{\K\L}|^{1/d_{\K\L}}\le
|\Delta_{\K}|^{1/d_{\K}}|\Delta_{\L}|^{1/d_{\L}}
\end{equation}
(see, for example, Proposition 4.9 of \cite{Nar}). Note first that
since the inequality
$$
|\Delta_{\K_k}|=a^{k-2}(ak^k+(k-1)^{k-1})<T_L^k
$$
holds for all sufficiently large $L$, it follows that
$|\Delta_{\K_k}|^{1/k}\le T_L$. Since ${\overline{\K_{k}}}$ is the
compositum of the $k$ linearly disjoint fields ${\K_k}^{(j)}$ each
of degree $k$ for $j=1,\ldots,k$ all having the property that
$|\Delta_{\K_{k}^{(j)}}|^{1/k}\le T_L$, repeated applications of
inequality \eqref{eq:delta} give
$|\Delta_{\overline{\K_k}}|^{1/k!}\le T_L^k$. Finally, since $\M$ is
the compositum of the linearly disjoint fields ${\overline{{\K_k}}}$
for $k=2,\ldots,L$ of degrees $k!$, respectively, repeated
applications of inequality \eqref{eq:delta} once more give that
\begin{eqnarray}
\label{eq:discriminant} |\Delta_{\M}|^{1/d_{\M}} & \le &
\prod_{k=2}^L |\Delta_{\overline{\K_k}}|^{1/k!}\le \prod_{k=2}^L
T_L^k<T_L^{L(L+1)/2}\nonumber\\
& = & \exp(((\log 2)/3+o(1))L^5),\qquad {\text{\rm as}}~L\to\infty.
\end{eqnarray}
In particular, the inequality
\begin{equation}
\label{eq:DeltaM} |\Delta_{\M}|\le \exp\(d_{\M}L^5\)
\end{equation}
holds once $L$ is sufficiently large. Assume now that
$\varepsilon\in (0,1/20)$, and that
\begin{equation}
\label{eq:1} (1/2+\varepsilon)L^2\log L<\log\log x.
\end{equation}
We then get that if $x>x_{\varepsilon}$, then
$$
d_{\M}L^5<\frac{\log x}{\log\log x},
$$
which in turn implies, via inequality \eqref{eq:DeltaM}, that
$$
x>|\Delta_{\M}|^{\log\log x}.
$$
Hence, inequality \eqref{eq:initial} together with the fact that
$\#G\le d_{\M}\ll \log |\Delta_{\M}|$, gives that
\begin{equation}
\label{eq:low} \pi_{\cal C}(x,\M/\Q)\ge x^{1/5-\varepsilon}.
\end{equation}
So, we see from \eqref{eq:1} that if we take
$$
L=\lfloor (2-\varepsilon/3)(\log\log x/\log\log\log
x)^{1/2}\rfloor-\delta,
$$
where $\delta\in \{0,1\}$ is chosen in such a way so that $L$ is
even, then inequality \eqref{eq:1} (and hence, inequality
\eqref{eq:low} also) is satisfied when $x$ is sufficiently large.

\medskip

We now discard the subset ${\cal Q}$ of primes $p\le x$ such that
$p\mid r^{M}-r^{M-1}-a$ for some prime $r\le x^{1/5-2\varepsilon}$
and some $M\in [L,2L\log x]$. Fixing $r$ and $M$, the number of such
primes is $\le \omega(|r^M-r^{M-1}-a|)$ (note that the integer
$r^M-r^{M-1}-a$ is not zero since $f_M(X)\in\Q[X]$ is irreducible
for all $M\ge 2$). Thus, the number of such choices is
$$
<\max\{\log(r^M),\log a\}< \max\{M\log x,\log T_{L}\}<2L(\log x)^2
$$
once $x$ is large. Summing this up over all the
$\pi(x^{1/5-2\varepsilon})$ choices of the prime $r\le x$ and over
all the $\lfloor 2L\log x\rfloor-L+1$ choices for $M$, we have that
$$
\#{\cal Q}<4\pi(x^{1/5-2\varepsilon})L^2(\log
x)^3<200x^{1/5-2\varepsilon}(\log x)^2(\log\log x)^2
$$
possibilities for $p$ once $x$ is large. Thus, if
$x>x_{\varepsilon}$, then
$$
\pi_{\cal C}(x,\M/\Q)-\#{\cal Q}\ge
x^{1/5-\varepsilon}-200x^{1/5-2\varepsilon}(\log x)^2(\log\log
x)^2>x^{1/5-2\varepsilon},
$$
so in particular there are such primes $p>x^{1/5-2\varepsilon}$
which are not in ${\cal Q}$.

\medskip

Let $p$ be one such prime. Look at congruence $a\pmod m$, where
$m=12p$. Assume that $\phi(n)\equiv a\pmod m$. Since $2\| a$ and $m$
is a multiple of $4$, it follows that either $n=4$ or $n=r^{\ell}$
or $2r^{\ell}$ for some odd prime $r$ and positive integer $\ell$.
If $n=4$, we get $2\equiv a\pmod p$, therefore $p\mid 2(q-1)$, which
is impossible since
$$
2(q-1)\le a\le T_L=\exp(O(\log\log x)^3)=x^{o(1)}
$$
as $x\to\infty$, while $p\ge x^{1/5-2\varepsilon}$. Thus,
$n=r^{\ell}$ or $2r^{\ell}$, therefore
$\phi(n)=r^{\ell}-r^{\ell-1}$. Hence, $r^{\ell}-r^{\ell-1}-a\equiv
0\pmod m$. If $\ell=1$, we get that $r\equiv a+1\pmod m$. Since
$a+1=2q+1\equiv 0\pmod 3$ and $3\mid m$, it follows that $3\mid r$,
and since $r$ is prime we get that $r=3$. Thus, $m\mid a-2$, leading
again to $p\mid 2(q-1)$, which is impossible. Thus, $\ell\ge 2$.
Since $r^{\ell}-r^{\ell-1}-a\equiv 0\pmod p$ with some $\ell\ge 2$,
it follows, from the way $p$ was chosen, that $\ell\ge L$ because
$f_k(X)\pmod p$ has no root modulo $p$ for any $k=2,\ldots,L-1$.
Thus, $\ell\ge L$. Since also $p\not \in {\cal Q}$, we get that
either $\ell\ge 2L\log x$, therefore
$$
n\ge r^{\ell}\ge 2^{2L\log x}=x^{(2\log 2)L}>m^{L}
$$
for large $x$, or $\ell\in [L,2L\log x]$, in which case
$r>x^{1/5-2\varepsilon}$, therefore
$$
r^{\ell}\ge x^{(1/5-2\varepsilon)L}>m^{(1/5-3\varepsilon)L}
$$
once $x$ is sufficiently large. Since $\varepsilon>0$ is arbitrary,
this shows that the smallest $n$ such that $\phi(n)\equiv a\pmod m$
satisfies indeed inequality \eqref{eq:maineq1} as $m\to\infty$,
provided that it exists.

\medskip

It remains to show that the progression $a\pmod m$ contains
totients. Well, from the way we choose $p$, the equation
$X^L-X^{L-1}\equiv a\pmod p$ has a solution modulo $p$ (in fact, $L$
distinct ones). Since $p$ and $a$ are coprime, it follows that any
solution $X\equiv x_0\pmod p$ of the above congruence has the
property that $x_0$ is not a multiple of $p$. Let $r$ be a prime
such that $r\equiv x_0\pmod p$. Then, $r^{L}-r^{L-1}-a$ is a
multiple of $p$. Imposing also that $r\equiv 3\pmod 4$, we get that
$r^{L}-r^{L-1}-a=r^{L-1}(r-1)-2q$ is also a multiple of $4$.
Finally, choosing $r\equiv 2\pmod 3$, since $L$ is even, we get
$r^L-r^{L-1}-a\equiv 1-2-2q\equiv 0\pmod 3$, because $q\equiv 1\pmod
3$. Thus, it is enough to choose primes $r$ such that $r\equiv
x_0\pmod p,~r\equiv 3\pmod 4$ and $r\equiv 2\pmod 3$. The above
system of congruences is solvable by the Chinese Remainder Lemma and
its solution is a congruence class modulo $m=12p$ which is coprime
to $m$. This class contains infinitely many primes $r$ by
Dirichlet's theorem on primes in arithmetical progressions and if
$r$ is any such prime then putting $n=r^L$ we have that the totient
$\phi(n)$ is indeed congruent to $a\pmod m$.

\medskip

This completes the proof of Theorem \ref{thm:1}. \qed

\medskip

{\bf Remark.} Under the Generalized Riemann Hypothesis for the
fields $\M$ constructed in the previous proof, we have
\begin{equation}
\label{eq:Chebotarev1} \left|\pi_{\cal C}(x,\M/\Q)-\frac{\#{\cal
C}}{\#{G}}{\text{\rm li}}(x)\right|\ll \frac{\#{\cal
C}}{\#G}x^{1/2}\log(|\Delta_{\M}|x^{d_{\M}})+\log|\Delta_{\M}|
\end{equation}
(see \cite{LO}). With our bound \eqref{eq:discriminant}, we have
$$
\log(|\Delta_{\M}|x^{d_{\M}})\ll d_{\M}(L^5+\log x).
$$
Inequality \eqref{eq:degree} shows that if $\varepsilon\in (0,1/2)$
and the inequality
$$
(1/2+\varepsilon)L^2\log L<(\log x)/2,
$$
holds, then estimate \eqref{eq:Chebotarev1} leads to
$$
\pi_{\cal C}(x,\M/\Q)\ge \frac{\pi(x)}{2\#G}\ge
x^{1/2-10\varepsilon}
$$
provided that $x>x_{\varepsilon}$. We apply again our previous
argument except that the set ${\cal Q}$ is now taken to be the set
of all primes $p$ that divide $r^{M}-r^{M-1}-a$ for some $M\in
[L,2L\log x]$ and some prime $r\le x^{1/2-20\varepsilon}$ and
conclude that there exist primes $p>x^{1/2-10\varepsilon}$ which are
not in ${\cal Q}$. Using again the fact that $\varepsilon\in
(0,1/2)$ can be chosen to be arbitrarily small, our previous
arguments now easily lead to the conclusion that the better
inequality \eqref{eq:maineq2} holds in this case with
$c_1=2^{-1/2}$. We give no further details.

\section{Proof of Theorem \ref{thm:2}}

If $a=0$, we take $n$ to be the first prime in the arithmetical
progression $1\pmod m$. By Heath-Brown's work on the Linnik constant
\cite{HB} we know that $n\ll m^{5.5}$. Clearly,
$\lambda(n)=n-1\equiv 0\pmod m$. From now on, we assume that $1\le
a\le m-1$. If $m$ is odd, we replace $m$ by $2m$ and $a$ by the even
number among $a$ and $a+m$ (note that since $m$ is odd, it follows
that $a$ and $a+m$ have different parities). Let $a=2^{\alpha}a_0$,
where $\alpha\ge 0$ and $a_0$ is odd. We replace $m$ by
$M=2^{\alpha}m$. Writing $\nu_2(k)$ for the exponent of $2$ in the
factorization of $k$, we have that $1\le \nu_2(k)<\nu_2(M)$. Let
$d=\gcd(a,M)$. Note that $2^{\alpha}\mid d$ and that
$d/2^{\alpha}\mid m$. Put $a_1=a/d,~m_1=M/d$ and note that $d$ is
even, $a_1$ is odd and $m_1$ is even. We now search for $n=p_1p_2$,
where $p_1,~p_2$ are primes of the form
$p_1=d\lambda_1+1,~p_2=d\lambda_2+1$ and $\gcd(p_1-1,p_2-1)=d$. In
this case, $\lambda(n)=d\lambda_1\lambda_2$, and now the congruence
$\lambda(n)\equiv a\pmod M$ is equivalent to
\begin{equation}
\label{eq:lambda} \lambda_1\lambda_2\equiv a_1\pmod {m_1}.
\end{equation}
Since $p_1$ and $p_2$ need to be prime numbers, by Dirichlet's
theorem on primes in arithmetical progressions, it suffices for
their existence that aside from the congruence \eqref{eq:lambda},
the conditions $\gcd(d\lambda_1+1,m_1)=\gcd(d\lambda_2+1,m_1)=1$ are
also fulfilled.

\medskip
We consider three different possibilities for the prime divisors of
$m_1$. First let $p\mid d$ and assume that $p^{\beta}\| {m_1}$ for
some $\beta\ge 1$.  In this case the conditions
$\gcd(d\lambda_1+1,m_1)=\gcd(d\lambda_2+1,m_1)=1$ are automatically
satisfied for any choices. Moreover, since $p$ does not divide
$a_1$, we can choose $\lambda_1\equiv a_1\pmod {p^{\beta}}$ and
$\lambda_2\equiv 1\pmod {p^{\beta}}$. Note that $p=2$ is such a
prime.

\medskip

Assume next that $p\nmid d$ and let $p^{\beta}\| {m_1}$ again for
some $\beta\ge 1$. Note that both $a_1$ and $d$ are invertible
modulo $p$ and that $p$ is odd. If $a_1d^2\not\equiv -1\pmod p$,
then we choose $\lambda_1\equiv a_1d\pmod {p^{\beta}}$ and
$\lambda_2\equiv d^{-1}\pmod {p^{\beta}}$. Then certainly
$\lambda_1\lambda_2\equiv a_1\pmod {p^{\beta}}$. Furthermore,
$\lambda_1d+1\equiv a_1d^2+1\not\equiv 0\pmod p$, therefore $p$ does
not divide $\lambda_1d+1$. Similarly, $\lambda_2d+1\equiv
2\not\equiv 0\pmod p$, because $p$ is odd, so $\lambda_2d+1$ is
coprime to $p$.

\medskip

Finally, assume that $a_1d^2\equiv -1\pmod p$. Let $\rho_p$ be some
primitive root modulo $p^{\beta}$, which exists since $p$ is odd,
and take $\lambda_1\equiv \rho_p d^{-1}\pmod {p^{\beta}}$, and
$\lambda_2=a_1d\rho_p^{-1}\pmod {p^{\beta}}$. Clearly,
$\lambda_1\lambda_2\equiv a_1\pmod {p^{\beta}}$. Furthermore,
$\lambda_1d+1\equiv \rho_p+1\not\equiv 0\pmod p$, because $-1$ is
not a primitive root modulo $p$. Similarly, $\lambda_2d+1\equiv
a_1d^2\rho_p^{-1}+1\equiv -\rho_p^{-1}+1\pmod p$, and this last
congruence class is not zero since $1$ is not a primitive root
modulo $p$.

Hence, for all prime powers $p^{\beta}$ dividing $m_1$ we have
constructed congruence classes $\lambda_1$ and $\lambda_2$ modulo
$p^{\beta}$ such that the congruence \eqref{eq:lambda} holds modulo
$p^{\beta}$ and $\lambda_1d+1,~\lambda_2d+1$ are not multiples of
$p$. By Dirichlet's theorem on primes in arithmetical progressions,
we can choose two distinct primes $p_1$ and $p_2$ such that if we
put $n=p_1p_2$, then indeed $\lambda(n)\equiv a\pmod m$.
Furthermore, by Heath-Brown's result mentioned earlier, we can
choose both $\lambda_1$ and $\lambda_2$ such that
$\max\{\lambda_1,\lambda_2\}\ll m_1^{5.5}$. This shows that
$$
\max\{p_1,p_2\}\ll m_1^{5.5}d\ll
\left(\frac{2^{\alpha}m}{d}\right)^{5.5}d \ll 2^{\alpha}
m^{5.5}\left(\frac{2^{\alpha}}{d}\right)^{4.5}\ll m^{6.5},
$$
leading to $n=p_1p_2\ll m^{13}$.

\end{document}